\documentclass[12pt]{amsart}
\usepackage{amssymb}
\usepackage{mathptmx}

\usepackage[latin1]{inputenc}

\newcommand{\mm}{\mathfrak m}
\newcommand{\C}{\mathbb{C}}
\newcommand{\N}{\mathbb{N}}

\newcommand{\R}{\mathbb{R}}

\newcommand{\Z}{\mathbb{Z}}
\newcommand{\Ac}{\mathcal{A}}

\newcommand{\Ic}{\mathcal{I}}
\newcommand{\Mcc}{\mathcal{M}}

\newcommand{\vb}{{\bf v}}

\DeclareMathOperator{\cl}{cl}

\DeclareMathOperator{\cx}{cx}
\DeclareMathOperator{\depth}{depth}

\DeclareMathOperator{\Ext}{Ext}

\DeclareMathOperator{\gin}{gin}

\DeclareMathOperator{\Hom}{Hom}

\DeclareMathOperator{\ini}{in}

\DeclareMathOperator{\Ker}{Ker}

\DeclareMathOperator{\nbc}{nbc}

\DeclareMathOperator{\reg}{reg}

\DeclareMathOperator{\supp}{supp}

\DeclareMathOperator{\Tor}{Tor}

\DeclareMathOperator{\tensor}{\otimes}

\DeclareMathOperator{\Dirsum}{\bigoplus}

\DeclareMathOperator{\pnt}{\raise 0.5mm \hbox{\large\bf.}}

\DeclareMathOperator{\mlpnt}{\!\!\hbox{\large\bf.}}
\DeclareMathOperator{\lpnt}{\hbox{\large\bf.}}

\newcommand{\fall}{\mbox{for all} ~}

\newtheorem{thm}{\bf Theorem}[section]
\newtheorem{lem}[thm]{\bf Lemma}
\newtheorem{cor}[thm]{\bf Corollary}
\newtheorem{prop}[thm]{\bf Proposition}

\theoremstyle{definition}
\newtheorem{defn}[thm]{\bf Definition}

\newtheorem{rem}[thm]{\bf Remark}
\newtheorem{ex}[thm]{\bf Example}

\title{Homological properties of Orlik-Solomon algebras}

\author{Gesa K\"ampf}
\address{Universit\"at Osnabr\"uck, Institut f\"ur Mathematik, 49069 Osnabr\"uck, Germany}
\email{gkaempf@uos.de}

\author{Tim R\"omer}
\address{Universit\"at Osnabr\"uck, Institut f\"ur Mathematik, 49069 Osnabr\"uck, Germany}
\email{troemer@uos.de}

\begin{document}

\begin{abstract}
The Orlik-Solomon algebra of a matroid can be considered as a quotient
ring over the exterior algebra $E$. At first we study homological properties
of $E$-modules
as e.g.\ complexity, depth and regularity.
In particular, we consider modules with linear injective resolutions.
We apply our results to Orlik-Solomon algebras of matroids
and give formulas for the complexity, depth  and
regularity of such rings in terms of invariants of the matroid. Moreover,
we characterize those matroids whose Orlik-Solomon ideal has a linear
projective resolution and compute in these cases the Betti numbers of the ideal.
\end{abstract}

%\subjclass{}
%\keywords{}

\maketitle

%------------------------------------------------------------------------
%
%
%
%------------------------------------------------------------------------
\section{Introduction}
\label{intro}
Let $\Ac=\{H_1,\dots,H_n\}$ be an essential central affine
hyperplane arrangement in $\C^m$, $X$ its complement and $K$ a
field. We choose linear forms $\alpha_i \in (\C^m)^*$ such that
$\Ker \alpha_i=H_i$ for $i=1,\dots,n$. Let $E=K\langle e_1,\dots,e_n
\rangle$ be the standard graded exterior algebra over $K$ where
$\deg e_i=1$ for $i=1,\dots,n$ and $\mm=(e_1,\ldots,e_n)$. For $S=\{j_1,\dots,j_t\}\subseteq
[n]=\{1,\dots,n\}$ we set $e_S=e_{j_1}\wedge \dots \wedge e_{j_t}$.
Usually we assume that $1\leq j_1<\dots< j_t\leq n$.
The elements $e_S$ are called \emph{monomials} in $E$.
It is
well-known that the singular cohomology $H^{\pnt}(X;K)$ of $X$ with
coefficients in $K$ is isomorphic to $E/J$ where $J$ is the
\emph{Orlik-Solomon ideal} of $X$ which is generated by all
\begin{equation}
\label{generators}
\partial{e_S}
= \sum_{i=1}^t (-1)^{i-1}  e_{j_1}\wedge \dots\wedge \widehat{e_{j_i}}
\wedge\dots\wedge e_{j_t} \text{ for } S=\{j_1,\dots,j_t\}\subseteq
[n]
\end{equation}
where $\{H_{j_1},\dots,H_{j_t}\}$ is a dependent set of hyperplanes
of $\Ac$, i.e.\ $\alpha_{j_1},\dots,\alpha_{j_t}$ are linearly
dependent. The algebra $E/J$ is also known as the
\emph{Orlik-Solomon algebra} of $X$. In the last decades many
researchers have studied the relationship between ring properties of
$E/J$ and properties of $\Ac$. See, e.g., the book of Orlik-Terao
\cite{ORTE} and the survey of Yuzvinsky \cite{Yu} for details.

Note that the definition of $E/J$ does only depend on the matroid of
$\Ac$ on $[n]$. For an arbitrary matroid on
$[n]$ the Orlik-Solomon algebra $E/J$ is defined as in the case of
hyperplane arrangements, i.e.\ $J$ is the ideal generated by
all $\partial{e_S}$ defined as in (\ref{generators}) where $S\subseteq
[n]$ is a dependent set of the given matroid. We are in particular
interested to investigate (co-)homological properties of
Orlik-Solomon algebras as modules over $E$. See, e.g., \cite{DY,
EPY, PE03,SS02, SS05} for related results.

In the first part of the paper we consider
arbitrary graded modules over the exterior algebra and
we study several algebraic and homological invariants of
such modules. In the second part of the paper we
apply these results to Orlik-Solomon algebras of matroids on $[n]$.

Let $\Mcc$ be the category of finitely generated
graded left and right $E$-modules $M$ satisfying $am=(-1)^{\deg
a\deg m}ma$ for homogeneous elements $a\in E$, $m\in M$.
For example if $J\subseteq E$ is a graded ideal, then $E/J$ belongs to $\Mcc$.

Let $M\in \Mcc$. Following \cite{AAH} we call
an element $v\in E_1$ \emph{regular} on $M$ (or
$M$\emph{-regular}) if the annihilator $0:_M v$ of $v$ in $M$ is the smallest
possible, that is, the submodule $vM$. An $M$\emph{-regular
sequence} is a sequence $v_1,\ldots,v_s$ in $E_1$ such that $v_i$ is
$M/(v_1,\ldots,v_{i-1})M$-regular for $i=1,\ldots,s$ and
$M/(v_1,\ldots,v_s)M \neq 0$. Every $M$-regular sequence can be
extended to a maximal one and all maximal regular sequences have the
same length. This length is called the \emph{depth} of $M$ over $E$
and is denoted by $\depth  M$.

For $i \in \N$ and $j \in \Z$ we call
$\beta_{i,j}(M)=\dim_K\Tor^E_i(K,M)_{j}$  the \emph{graded Betti numbers} and
$\mu_{i,j}(M)=\dim_K\Ext^i_E(K,M)_j$ the \emph{graded Bass numbers} of $M$.
Recall that $M \in \Mcc$ has a \emph{$d$-linear (projective) resolution}
if $\beta_{i,i+j}(M)=0$ for all $i$ and $j\neq d$.
We say that  $M \in \Mcc$ has a \emph{$d$-linear injective resolution}
if $\mu_{i,j-i}(M)=0$ for all $i$ and $j\neq d$.
(See Section \ref{sect Cartan complex} for reformulations of this definitions.)
The \emph{complexity} of $M$ measures the growth rate of the
Betti numbers of $M$ and is defined as
$$
\cx  M=\inf\{c \in \N: \beta_i(M)\leq \alpha i^{c-1} \fall i\geq 1, \alpha \in\R\}
$$
where
$\beta_i(M)=\sum_{j\in\Z}\beta_{i,j}(M)$ is the $i$-th
total \emph{Betti number} of $M$.

Aramova, Herzog and Hibi \cite{AHH} showed that analogously to the situation
in a polynomial ring Gr\"obner basis theory
can be developed over $E$. Especially generic initial
ideals can be constructed.
In the following
the monomial order considered on $E$ is always the reverse
lexicographic order induced by $e_1>\dots>e_n$. Let
$\ini(J)$ denote the initial ideal and
$\gin(J)$ denote the generic initial ideal of a graded ideal
$J\subseteq E$. For all
results related to generic initial ideals we assume that $|K|=\infty$.
After some definitions and general remarks in Section \ref{sect Cartan complex}
we consider in Section \ref{sect gin} the ideal $\gin(J)$ and
study relations between $E/J$ and $E/\gin(J)$.
In \cite{AHH} it is observed that
$\beta_{i,j}(E/J) \leq \beta_{i,j}(E/\ini(J))$ for all $i,j$.
In Corollary \ref{mu gin} we show that also
$$
\mu_{i,j}(E/J) \leq \mu_{i,j}(E/\ini(J)) \text{ for all } i,j.
$$
Herzog and Terai proved in \cite[Proposition 2.3]{HT}
that $\depth  E/J=\depth E/\gin(J)$ and $\cx  E/J=\cx E/\gin(J).$
These numbers
can be computed in terms of
combinatorial data associated to a generic initial ideal.
More precisely,
let $\supp(u)=\{i\in [n]:e_i|u\}$ and $\max(u)=\max\supp(u)$
for a monomial $u$ of $E$. Similar we define
$\min(u)=\min\supp(u)$.
A direct consequence of
Proposition \ref{cx}  and
Proposition \ref{d(EJ)} is that
\begin{eqnarray*}
\cx  E/\gin(J)&=&\max\{\max(u):u\in G(\gin(J))\},\\
d(E/\gin(J))&=&n-\max\{\min(u):u\in G(\gin(J))\}
\end{eqnarray*}
where
$G(\gin(J))$ denotes the unique minimal set of monomial
generators of $\gin(J)$
and
$d(M)=\max\{i\in \Z:M_i\not=0\}$ for $M \in \Mcc$.
Using the formula $\cx  M +\depth  M=n$ (see \cite[Theorem 3.2]{AAH})
we get also an expression for $\depth  E/\gin(J)$.

In Section \ref{sect depth} we present some results related to $\depth M$.
Let
$H(M,t)=\sum_{i\in\Z}\dim_K M_i t^i$
denote the \emph{Hilbert series} of $M$.
Then depth of $E/J$ where $J\subseteq E$ is a graded ideal and $E/J$ has a linear injective resolution
can be computed as follows.
We show in Theorem \ref{hilbert series}
that if
$|K|=\infty$, $E/J$ has a linear injective resolution and $\depth  E/J=s$,
then there exists a polynomial $Q(t)\in\Z[t]$ with non-negative coefficients
such that
$$
H(E/J,t)=Q(t)\cdot(1+t)^s \text{ and } \ Q(-1)\not=0.
$$
Observe that it is not possible to generalize this equation
in this form to the case of arbitrary quotient rings over $E$.

The $K$-algebra $E$ is injective and
thus $(\cdot)^*=\Hom_E(\cdot,E)$ is an exact functor on $\Mcc$.
By \cite[Proposition 5.2]{AHH} we know that
$\mu_{i,j}(M)=\beta_{i,n-j}(M^*)$ for all $i,j$.
In particular, we see that $M$ has a $d$-linear projective resolution
if and only if $M^*$ has an $(n-d)$-linear injective resolution.
In Theorem \ref{regdual} we observe that
additionally
$$
\depth  M=\depth  M^* \text{ and }
\cx M=\cx  M^*.
$$

Let $\Delta$ be a simplicial complex on $[n]$, i.e.\
$\Delta$ is a set of subsets of $[n]$ and
if $F \subseteq G$ for some  $G\in \Delta$, then
we also have $F \in \Delta$.
The \emph{exterior face ring} of $\Delta$ is $E/J_\Delta$ where
$J_\Delta=(e_F : F \subseteq [n], F\notin \Delta )$.
Then it is easy to see that $E/J_\Delta$ has a linear injective resolution
if and only if $\Delta$ is a Cohen-Macaulay complex.
See Example \ref{ex face ring} for details.
A reformulation and generalization to the matroid case
of \cite[Theorem 1.1]{EPY}
is that the Orlik-Solomon algebra of a matroid
has always a linear injective resolution.
These examples motivate to study
in general modules with linear injective resolutions, which is done in Section \ref{sect linear injective}.

Recall that $\reg M= \max\{j-i : \beta_{i,j}(M)\neq 0\}$ for $0\neq M \in \Mcc$
is the \emph{regularity} of $M$.
We prove in Theorem \ref{reg+depth=d}
that the regularity  of a quotient ring $E/J$ with $d$-linear
injective resolution satisfies
$$
\reg E/J+\depth  E/J=d.
$$
In the remainder of Section \ref{sect linear injective}
we present several technical results related to modules with injective linear
resolutions which we need in Section \ref{sect os algebras}.

In Section \ref{sect os algebras} we investigate homological properties of
Orlik-Solomon algebras of matroids. For convenience of the reader we
start with all necessary matroid notions.
At first we present a compact proof of the
mentioned result of Eisenbud, Popescu and Yuzvinsky (see \cite[Theorem
1.1]{EPY}) that Orlik-Solomon algebras have a linear injective resolution.
We determine the depth and the regularity of an
Orlik-Solomon algebra in Theorem \ref{depth general} and Corollary \ref{reg=l-k}.
More precisely, if $|K|=\infty$ and
$J \subseteq E$ is the Orlik-Solomon ideal of a loopless
matroid on $[n]$ of rank $l$ with $k$ components,
then
$$
\depth  E/J = k \text{ and } \reg E/J =  l-k.
$$
Finally we characterize in Theorem
\ref{characterisation} those matroids whose Orlik-Solomon ideal has a
linear projective resolution: The Orlik-Solomon
ideal $J$ of a matroid has an $m$-linear projective resolution if and only
if the matroid satisfies one of the following three conditions:
\begin{enumerate}
\item The matroid has a loop and $m=0$.
\item The matroid has no loops, but non-trivial parallel classes,
$m=1$ and the matroid is $U_{1,n_1}\oplus\dots\oplus U_{1,n_k}\oplus U_{f,f}$ for some $k,f\geq 0$.
\item The matroid is simple and it is $U_{m,n-f}\oplus U_{f,f}$ for some $0\leq f\leq n$.
 \end{enumerate}
Here $U_{m,n}$ is the \emph{uniform matroid},
whose independent sets are all subsets of $[n]$ with $m$
or less elements. In Theorem  \ref{Betti_extreme_cases}
we give formulas for the total Betti numbers
of these Orlik-Solomon ideals.

We conclude the paper  with
examples of matroids with small rank or small number of elements
to which we apply our results.

\section{Preliminaries}
\label{sect Cartan complex}

In this section we
recall some definitions
and facts about the exterior algebra.
Let $M\in\Mcc$ with minimal graded free resolution
$$
\ldots\longrightarrow F_2\longrightarrow F_1 \longrightarrow F_0\longrightarrow 0.
$$
To distinguish it from the injective resolution of $M$, we
call this resolution also the projective one.
Because of the minimality the $i$-th free module in this resolution is
$F_i=\Dirsum_{j\in\Z}E(-j)^{\beta_{i,j}(M)}$.
We see that the resolution is $d$-linear (as defined in Section \ref{intro})
for some $d \in \Z$
if and only if it is of the form
$$
\ldots\longrightarrow E(-d-2)^{\beta_{2}(M)}\longrightarrow E(-d-1)^{\beta_{1}(M)} \longrightarrow E(-d)^{\beta_{0}(M)}\longrightarrow 0.
$$
This is equivalent to say that if we choose matrices for the maps in the resolution,
then all entries in these matrices are elements in $\mm=(e_1,\dots,e_n)$ of degree 1.

Next we consider for $M$ its minimal graded injective resolution
$$
0\longrightarrow I^0\longrightarrow I^1 \longrightarrow I^2\longrightarrow \ldots.
$$
Since $E$ is injective, we have $I^i=\Dirsum_{j\in\Z}E(n-j)^{\mu_{i,j}(M)}$.
Computing $\Ext^i_E(K,M)$ via the latter resolution shows that this resolution is $d$-linear if and only if it is of the form
$$
0\longrightarrow E(n-d)^{\mu_0(M)}\longrightarrow E(n-d+1)^{\mu_1(M)}\longrightarrow E(n-d+2)^{\mu_2(M)}\longrightarrow\ldots
$$
In particular, in this case $d(M)=\max\{i:M_i\not=0\}=d$
because the socle $0:_M \mm$ of $M$
is isomorphic to the socle of $E(n-d)^{\mu_0(M)}$ which lives
in degree $d$.

Let $M^*=\Hom_E(M,E)$. We call $M^*$ the \emph{dual} of $M$.
Note that the dual of a (minimal) graded projective resolution of $M$ is a
(minimal) graded injective resolution of $M^*$.

For a $K$-vector space $W$ let $W^\vee=\Hom_K(W,K)$ be the
$K$-dual of $W$. In  \cite[Proposition 5.1]{AHH} it was observed that
$
(M^*)_i\cong (M_{n-i})^\vee$ as $K$-vector
spaces.

A very useful complex over $E$ is the Cartan complex
which plays a similar role as the Koszul complex for the polynomial ring.
It is defined as follows. For a sequence
$\vb=v_1,\ldots,v_m\in E_1$  let
$C\mlpnt(\vb;E)=C\mlpnt(v_1,\ldots,v_m;E)$ be the free divided power
algebra $E\langle x_1,\ldots,x_m\rangle$. It is generated by the
divided powers $x_i^{(j)}$ for $i=1,\ldots,m$ and $j\geq 0$ which
satisfy the relations $x_i^{j}x_i^{k}=((j+k)!/(j!k!))x_i^{j+k}$.
Thus $C_i(\vb;E)$ is a free $E$-module with basis
$x^{(a)}=x_1^{(a_1)}\cdots x_m^{(a_m)}$, $a\in \N^m$, $|a|=i$. The
$E$-linear differential on $C\mlpnt(v_1,\ldots,v_m;E)$ is
$$\partial_i : C_i(v_1,\ldots,v_m;E)\longrightarrow C_{i-1}(v_1,\ldots,v_m;E), \qquad x^{(a)}\mapsto \sum_{a_j>0}v_jx_1^{(a_1)}\cdots x_j^{(a_j-1)}\cdots x_m^{(a_m)}.$$
One easily sees that $\partial\circ\partial=0$ so this is indeed a complex.

\begin{defn} Let $M \in \Mcc$.
The complexes
$$
C\mlpnt(\vb;M)=C\mlpnt(\vb;E)\tensor_EM \quad \text{ and } \quad C^{\lpnt}(\vb;M)=\Hom_E(C\mlpnt(\vb;E),M)
$$
are called the \emph{Cartan complex} and \emph{Cartan cocomplex} of
$\vb$ with values in $M$. The corresponding homology modules
$$
H_i(\vb;M)=H_i(C\mlpnt(\vb;M)) \quad \text{ and } \quad H^i(\vb;M)=H^i(C^{\pnt}(\vb;M))
$$
are called the \emph{Cartan homology} and \emph{Cartan cohomology}
of $\vb$ with values in $M$.
\end{defn}

The elements of $C^{i}(\vb;M)$ can be identified with homogeneous
polynomials $\sum m_ay^a$ with $m_a\in M$, $a\in \N^m$, $|a|=i$ and
$$
\sum m_ay^a(x^{(b)})=\begin{cases} m_a &  \text{if }a=b,\\
                                    0   & \text{if }a\not=b.\\
                      \end{cases}
$$
In particular, $C^{\pnt}(\vb;E)$ is just the
polynomial ring $S=K[y_1,\ldots,y_m]$. After this identifications
the differential on $C^{\pnt}(\vb;M)$ is simply the multiplication
with $\sum_{i=1}^mv_iy_i$.

Setting $\deg x_i=1$ and $\deg y_i=-1$ induces a grading on the
complexes and their homologies. Cartan homology and Cartan
cohomology are related as follows:

\begin{prop}
\label{co and hom}\cite[Proposition 4.2]{AH} Let $M\in\Mcc$ and
$\vb=v_1,\ldots,v_m \in E_1$. Then
$$
H_i(\vb;M)^*\cong H^i(\vb;M^*) \text{ as graded } E\text{-modules } \fall \ i\in \N.
$$
\end{prop}

Cartan (co)homology can be used inductively as there are long exact
sequences connecting the (co)homologies of $v_1,\ldots,v_j$ and
$v_1,\ldots,v_{j+1}$.

\begin{prop}
\label{long cartan seq}\cite[Propositions 4.1, 4.3]{AH}
Let
$M\in\Mcc$ and $\vb=v_1,\ldots,v_m \in E_1$. For all $j=1,\ldots,m$
there exist long exact sequences of graded $E$-modules
$$
\ldots\longrightarrow H_i(v_1,\ldots,v_j;M)\longrightarrow H_i(v_1,\ldots,v_{j+1};M)\longrightarrow H_{i-1}(v_1,\ldots,v_{j+1};M)(-1)
$$
$$
\longrightarrow H_{i-1}(v_1,\ldots,v_j;M)\longrightarrow H_{i-1}(v_1,\ldots,v_{j+1};M)\longrightarrow \ldots
$$
and
$$ \ldots \longrightarrow H^{i-1}(v_1,\ldots,v_{j+1};M)\longrightarrow H^{i-1}(v_1,\ldots,v_j;M)\longrightarrow H^{i-1}(v_1,\ldots,v_{j+1};M)(+1)$$
$$\stackrel{\cdot y_{j+1}}\longrightarrow H^i(v_1,\ldots,v_{j+1};M)\longrightarrow H^i(v_1,\ldots,v_j;M)\longrightarrow\ldots
$$
\end{prop}

It is well-known that the Cartan complex $C\mlpnt(v_1,\ldots,v_m;E)$
with values in $E$ is exact and hence it is the minimal graded free
resolution of $E/(v_1,\ldots,v_m)$ over $E$. Thus it can be used to
compute $\Tor_i^E(E/(v_1,\ldots,v_m),\cdot)$ and
$\Ext^i_E(E/(v_1,\ldots,v_m),\cdot)$:

\begin{prop}\cite[Theorem 2.2]{AHH}\label{TorExt and Cartan}
Let $M\in\Mcc$ and $\vb=v_1,\ldots,v_m\in E_1$. There are
isomorphisms of graded $E$-modules
$$
\Tor_i^E(E/(v_1,\ldots,v_m),M)\cong H_i(\vb;M), \qquad \Ext^i_E(E/(v_1,\ldots,v_m),M)\cong H^i(\vb;M).
$$
\end{prop}

Regularity of a sequence can be detected by its Cartan complex:
\begin{prop}\label{reg over cartan}\cite[Remark 3.4]{AAH}
Let $M\in\Mcc$ and $\vb=v_1,\ldots,v_m \in E_1$.
The following statements are equivalent:
\begin{enumerate}
 \item $\vb$ is $M$-regular;
 \item $H_1(\vb;M)=0$;
 \item $H_i(\vb;M)=0$ for $i\geq 1$.
\end{enumerate}
\end{prop}

In particular, permutations of regular sequences are regular
sequences because the vanishing of the first Cartan homology does
not depend on the order of the elements as one easily sees using
Proposition \ref{long cartan seq}.

\section{Initial and generic initial ideals}\label{sect gin}

In this section we describe some properties of generic initial
ideals and stable ideals. The existence of the generic initial ideal
$\gin (J)$ of a graded ideal $J$ in the exterior algebra over an
infinite field is proved by Aramova, Herzog and Hibi in
\cite[Theorem 1.6]{AHH}, analogously to the case of ideals in the
polynomial ring. (See, e.g., also \cite[Chapter 5]{Gr} or \cite{NRV} for related results.)

A monomial ideal $J\subseteq E$ is called \emph{stable} if
$e_j\frac{u}{e_{\max(u)}}\in J$ for every monomial $u\in J$ and
$j<\max(u)$. The ideal $J$ is called \emph{strongly stable} if
$e_j\frac{u}{e_i}\in J$ for every monomial $u\in J$, $i\in\supp(u)$
and $j<i$.

The generic initial ideal $\gin (J)$ of a graded ideal $J$ is
strongly stable if it exists (see, e.g., \cite[Proposition
1.7]{AHH}). This is independent of the characteristic of $K$ in
contrast to ideals in a polynomial ring. By (the proof of)
\cite[Lemma 1.1]{HT} we have:
\begin{lem}
\label{dual gin} Let $|K|=\infty$ and $J\subseteq E$ be a graded
ideal in $E$. Then
$$
\ini((E/J)^*)\cong (E/\ini(J))^*
$$
as graded $E$-modules, where $(E/J)^*$ is identified with the ideal
$0:_EJ$. In particular,
$$
\gin((E/J)^*)\cong (E/\gin(J))^*.
$$
\end{lem}

With this result we can compare the Bass numbers of a graded ideal
with the Bass numbers of its initial ideal because we already know
$\beta_{i,j}(E/J)\leq\beta_{i,j}(E/\ini(J))$ for all $i,j$ by
\cite[Proposition 1.8]{AHH}.

\begin{cor}\label{mu gin}
Let $J\subseteq E$ be a graded ideal. Then
$$
\mu_{ij}(E/J)\leq\mu_{ij}(E/\ini(J)) \text{ for all } i,j.
$$
\end{cor}
\begin{proof}
It follows from the inequalities
$$\beta_{i,j}(E/J)\leq\beta_{i,j}(E/\ini(J))$$
and Lemma \ref{dual gin} that
\begin{eqnarray*}
 \mu_{i,j}(E/J)&=&   \beta_{i,n-j}((E/J)^*)\\
               &\leq&\beta_{i,n-j}(\ini((E/J)^*))\\
               &=&   \beta_{i,n-j}((E/\ini (J))^*)\\
               &=&   \mu_{i,j}(E/\ini(J)).
\end{eqnarray*}
\end{proof}

In the following we collect some results on (strongly) stable
ideals. They are inspired by the chapter on squarefree strongly
stable ideals in the polynomial ring in \cite{HS}. Let $G(J)$ be the
unique minimal system of monomials generators of a monomial ideal
$J$.

Aramova, Herzog and Hibi \cite{AHH} computed  a formula for the
graded Betti numbers of stable ideals:

\begin{lem}\cite[Corollary 3.3]{AHH}
\label{betti}
Let $0\neq J\subset E$ be a stable ideal. Then
$$
\beta_{i,i+j}(J)=\sum_{u\in G(J)_j}\binom{\max(u)+i-1}{\max(u)-1} \ \ \fall i\geq 0, j\in\Z.
$$
\end{lem}

In particular, if $J$ is stable and generated in one degree, it has
a linear projective resolution. An example for such an ideal is the
maximal ideal $\mm$ of $E$ and all its powers.

The complexity of a stable ideal $J$ can be interpreted in terms of
$G(J)$.

\begin{prop}
\label{cx}
Let $0\neq J\subset E$ be a stable ideal. Then
$$
\cx  E/J=\max\{\max(u):u\in G(J)\}.
$$
\end{prop}
\begin{proof}
This is evident from the formula for the Betti numbers of stable
ideals since
$$
\beta_i(J)=\sum_{k=1}^n m_k(J)\binom{k+i-1}{k-1}
$$
where $m_k(J)=|\{u\in G(J):\max(u)=k\}|$. The binomial coefficient in
this sum is a polynomial in $i$ of degree $k-1$ and the number
$\max\{\max(u):u\in G(J)\}$ is exactly the maximal $k$
for which  $m_k(J)\not=0$.
\end{proof}

Recall that
$$
d(M)=\max\{i\in \Z:M_i\not=0\}=n-\min\{i\in \Z: (M^*)_i\not=0\}.
$$
Here the second equality results from the isomorphism $(M^*)_i\cong (M_{n-i})^\vee$.
In the case of strongly stable ideals $J$ this number has a meaning in terms of $G(J)$.

\begin{prop}
\label{d(EJ)}
Let $0\neq J\subset E$ be a strongly stable ideal. Then
$$
d(E/J)=n-\max\{\min(u):u\in G(J)\}.
$$
\end{prop}
Observe that the right hand side of the equation
does not change when replacing $G(J)$ by
$J$ because $J$ is strongly stable.
\begin{proof}
Set $s=\max\{\min(u):u\in G(J)\}$. We want to show that
$$s=\min\{i: (E/J)^*_i\not=0\}.$$

As $J\subseteq (e_1,\ldots,e_s)$ we obtain ``$\geq$'' immediately
from the equivalence
$$ J\subseteq (e_{i_1},\ldots, e_{i_r}) \Leftrightarrow e_{i_1}\cdots e_{i_r}\in 0:_EJ\cong(E/J)^*$$
for $i_1,\ldots,i_r\in [n]$.

The other inequality ``$\leq$'' follows if we show that $J\subseteq
(e_{i_1},\ldots, e_{i_r})$  implies $r\geq s$.

First consider the case that $J\subseteq (e_s)$. As $J$ is strongly
stable, $e_i\frac{u}{e_s}\in J$ for all monomials $u\in J$ and all
$i<s$. But $e_i\frac{u}{e_s}\not\in (e_s)$ for $i\not\in\supp(u)$
and thus $i\in\supp(u)$ for all $i<s$. By the definition of $s$
this implies $s=1$ and hence $r\geq s$.

Now assume $J\not\subseteq (e_s)$ and consider the ideal
$\overline{J}=J+(e_s)/(e_s)$. This ideal is again strongly stable in
the exterior algebra $\overline{E}$ in $n-1$ variables
$e_1,\ldots,e_{s-1},e_{s+1},\ldots,e_n$. The position of $e_i$
diminishes by one for $i>s$. We see that
$$\min(\overline{u})=\begin{cases}
           \min(u)   & \text{if }\min(u)<s\\
           \min(u)-1 & \text{if }\min(u)>s\\
          \end{cases}$$
for the residue class of a monomial $u$ of $E$ with $e_s\!\not|\;u$ . By
the choice of $s$ we see immediately that $\max\{\min(\overline{u}):\overline{u}\in
\overline{J}\}=s$. On the other hand $\overline{J}\subseteq
(e_{i_1},\ldots,e_{i_r})+(e_s)/(e_s)$.
By an appropriate  induction on $n$ we  get that
$s\leq r$ if $s\not\in\{i_1,\ldots,i_r\}$ and $s\leq r-1$ otherwise.
\end{proof}

\section{Depth of graded $E$-modules}
\label{sect depth}

The purpose of this section is to present further results on
regular sequences over the exterior algebra.

Recall that $H(M,t)=\sum_{i\in\Z}\dim_K M_i t^i$
denotes the Hilbert series of  a graded $E$-module $M$.
Analogously to the well-known Hilbert-Serre theorem
(see, e.g., \cite[Proposition 4.4.1]{BH})
we have the following result.

\begin{thm}\label{hilbert series}
Let $|K|=\infty$ and $0\neq J\subset E$ be a graded ideal with $\depth E/J=s$. Let $E/J$ have a linear injective resolution.
Then there exists a polynomial $Q(t)\in\Z[t]$ with non-negative
coefficients such that the Hilbert series of $E/J$ has the form
$$
H(E/J,t)=Q(t)\cdot (1+t)^s \text{ with } \ Q(-1)\not=0.
$$
\end{thm}

Note that it is not possible to generalize the equation in this
form to the case of arbitrary quotient rings. The ideal $(e_1e_2,e_1e_3,e_1e_4,e_2e_3e_4)$  provides a counterexample.
%In particular, this equality does not follow from \cite[Corollary 3.8]{AAH}.

\begin{proof}
Let $M=E/J$.
First of all we show that if $v$ is $M$-regular, then
\begin{equation}
\label{eq1}
H(M,t)=(1+t)H(M/vM,t).
\end{equation}

We have the exact sequence
$$0\longrightarrow vM\longrightarrow M\longrightarrow M/vM\longrightarrow 0$$
which implies
\begin{equation}
\label{eq2} H(vM,t)=H(M,t)-H(M/vM,t).
\end{equation}
As $v$ is
$M$-regular the sequence
$$ 0\longrightarrow vM(-1)\longrightarrow M(-1)\stackrel{\cdot v}{\longrightarrow} M \longrightarrow M/vM \longrightarrow 0$$
is exact and gives
\begin{equation}
\label{eq3} (1-t)H(M,t)=H(M/vM,t)-tH(vM,t).
\end{equation}
Equations (\ref{eq2}) and (\ref{eq3}) together show (\ref{eq1}).

Thus if $v_1,\ldots,v_s$ is a maximal $E/J$-regular sequence, we
obtain inductively
$$H(E/J,t)=(1+t)^s H(E/J+(v_1,\ldots,v_s),t).$$
The Hilbert series of $E/(J+(v_1,\ldots,v_s))$ is a polynomial with
nonnegative coefficients and $\depth  E/(J+(v_1,\ldots,v_s))=0$. We
claim that the polynomial $1+t$ does not divide
$H(E/(J+(v_1,\ldots,v_s)),t)$.

To this end we may assume that $\depth  E/J=0$. The Hilbert series
and the depth of $E/J$ and $E/\gin(J)$ coincide, so we may assume in
addition that $J$ is strongly stable. Then we know the Betti numbers
of $J$. Proving that $1+t$ does not divide the Hilbert series of $E/J$
is the same as showing this for $J$ as the Hilbert series of $E$
is $(1+t)^n$.

Let $m_{kj}(J)=|\{u\in G(J):\max(u)=k, \deg(u)=j\}|$. Computing the
Hilbert series of $J$ via the minimal graded free resolution of $J$
gives
\begin{eqnarray*}
H(J,t)&=&\sum_{i\geq 0}(-1)^iH(\Dirsum_{j\in\Z}E(-j)^{\beta_{ij}(J)},t)\\
      &=&\sum_{i\geq 0}(-1)^i\sum_{j\in\Z}t^j\beta_{ij}(J)(1+t)^n\\
      &=&\sum_{i\geq 0}(-1)^i\sum_{j\in\Z}t^{i+j}\beta_{i,i+j}(J)(1+t)^n\\
      &=&\sum_{i\geq 0}(-1)^i\sum_{j\in\Z}t^{i+j}(1+t)^n\sum_{u\in G(J)_j}\binom{\max(u)+i-1}{\max(u)-1}\\
      &=&\sum_{i\geq 0}(-1)^i\sum_{k=1}^n\sum_{j=1}^k m_{kj}(J)\binom{k+i-1}{k-1}t^{i+j}(1+t)^n\\
      &=&\sum_{k=1}^n\sum_{j=1}^km_{kj}(J)t^j(1+t)^n\sum_{i\geq 0}\binom{k+i-1}{k-1}(-1)^it^i\\
      &=&\sum_{k=1}^n\sum_{j=1}^km_{kj}(J)t^j(1+t)^n \frac{1}{(1+t)^k}\\
      &=&\sum_{k=1}^n\sum_{j=1}^km_{kj}(J)t^j(1+t)^{n-k}.\\
\end{eqnarray*}
All coefficients appearing in the last sum are non-negative
hence no term can be canceled by another.
$n-\cx E/J=\depth E/J=0$ and Proposition \ref{cx} imply that $m_{nj}(J)\not=0$ for
some $j$. Let $u=e_Fe_n\in G(J)$. We have
$e_Fe_i\in J$ for all $i=1,\ldots,n$ because $J$ is stable. The dual of $E/J$ is
$(E/J)^*\cong 0:_E J$, which is generated by all monomials $e_F$ with
$e_{F^c}\not\in J$ (cf. Example \ref{ex face ring}). But then
$e_{(F\cup\{i\})^c}=e_{F^c\setminus\{i\}}\not\in (E/J)^*$ for all
$i\not\in F$. As $e_F\not\in J$ (otherwise $e_Fe_n$ would not be a
minimal generator), the complement $e_{F^c}$ is in $(E/J)^*$ and
even a minimal generator. The ideal $(E/J)^*$ has an $(n-d)$-linear
projective resolution, in particular it is generated in degree $n-d$, so
$|F|=n-|F^c|=n-(n-d)=d.$
Thus we have seen that every minimal generator 
$u\in G(J)$ with $n\in\supp(u)$ has degree $d+1$. 
Hence $m_{nj}(J)=0$ for $j\not=d+1$ and $m_{n,d+1}\not=0$. 
Thus there is exactly one summand in $H(J,t)$ that is not divisible by $1+t$ 
and we see that $1+t$ does not divide $H(J,t)$.
\end{proof}

Next we want to compare regular sequences on a module and its dual.
To this end we need the following lemma. Since $v^2=0$ for $v\in
E_1$ the multiplication map on a graded $E$-module $M$ induces a
complex
$$
(M,v): \qquad \ldots\longrightarrow M_{i-1}\stackrel{\cdot v }{\longrightarrow} M_{i}\stackrel{\cdot v }{\longrightarrow} M_{i+1}\longrightarrow \ldots
$$
The homology of this complex is denoted by $H_i(M,v)$. Then $v$ is
regular on $M$ if and only if $H_i(M,v)=0$ for all $i$.

\begin{lem}\label{regseq}
 Let $0\to U\to M\to N\to 0$ be an exact sequence of modules in $\Mcc$. If $v\in E_1$ is regular on two of the three modules, then it is regular on the third.
\end{lem}
\begin{proof}
The short exact sequence induces a short exact sequence of complexes
$$0\longrightarrow (U,v) \longrightarrow (M,v) \longrightarrow (N,v) \longrightarrow 0$$
which induces a long exact sequence of homology modules
$$
\ldots \longrightarrow H_{i-1}(N,v) \longrightarrow H_i(U,v) \longrightarrow H_i(M,v)
\longrightarrow H_i(N,v)\longrightarrow  H_{i+1}(U,v)\longrightarrow\ldots
$$
Then the observation that $v$ is regular on one of these modules, say $M$, if and only if the corresponding homology $H_i(M,v)$ is zero for all $i$ concludes the proof
\end{proof}

Let $v_1,\ldots,v_s\in E_1$ and $M\in\Mcc$. To simplify notation we
define $$\qquad \qquad \quad \quad  H_i(k)=H_i(v_1,\ldots,v_k;M)
\text{ for } i>0,\  k=1,\ldots,s$$ and
\begin{eqnarray*}
\tilde{H}_0(k)&=&\frac{0:_{M/(v_1,\ldots,v_{k-1})M}v_k}{v_k(M/(v_1,\ldots,v_{k-1})M)}.\\
\end{eqnarray*}

Analogously
$$\qquad \qquad \quad \quad  H^i(k)=H^i(v_1,\ldots,v_k;M) \text{ for } i>0,\  k=1,\ldots,s$$
and
\begin{eqnarray*}
\tilde{H}^0(k)&=&\frac{0:_{0:_M(v_1,\ldots,v_{k-1})}v_k}{v_k(0:_M(v_1,\ldots,v_{k-1}))}.\\
\end{eqnarray*}

Finally we set $H_i(0)=H^i(0)=0$ for $i>0$. The modules
$\tilde{H}_0(k)$ and $\tilde{H}^0(k)$ are not the $0$-th Cartan
homology and cohomology but defined such that the long exact
sequences of Cartan homology and cohomology modules of Proposition
\ref{long cartan seq} induces exact sequences
$$\ldots\longrightarrow H_2(k) \longrightarrow H_1(k)(-1) \longrightarrow H_1(k-1)\longrightarrow H_1(k)\longrightarrow \tilde{H}_0(k)(-1)\longrightarrow 0$$
and
$$0\longrightarrow \tilde{H}^0(k)(+1)\longrightarrow H^1(k)\longrightarrow H^1(k-1)\longrightarrow H^1(k)(+1) \longrightarrow H^2(k)\longrightarrow\ldots$$

\begin{thm}\label{regdual}
Let $0\neq M\in\Mcc$. A sequence $\vb=v_1,\ldots,v_s\in E_1$ is an
$M$-regular sequence if and only if it is an $M^*$-regular sequence.
In particular,
$$
\depth  M=\depth  M^* \text{ and }\cx M=\cx  M^*.
$$
\end{thm}
\begin{proof} We may assume that $|K|=\infty$.
It is enough to prove $\depth  M=\depth  M^*$.
Then $\cx M=\cx  M^*$ follows from the formula
$\cx M+\depth M=n$.

To prove the assertion it is enough to show that if $\vb$ is an
$M^*$-regular sequence, then it is an $M$-regular sequence as well.

First of all we state two observations which will be used several
times in the proof. Let $N,N'\in\Mcc$ and $v\in E_1$.
\begin{equation}\tag{$*$}
 \text{If } v \text{ is } N'\text{-regular and } vN'\subseteq vN, \text{ then } vN\cap N'=vN'.
\end{equation}
This is obvious since $x\in vN\cap N'$ implies $x\in 0:_{N'}v=vN'$.
\begin{equation}\tag{$**$}
 \text{If }v \text{ is regular on } N,N' \text{ and } N\cap N', \text{ then } v \text{ is regular on } N+N'.
\end{equation}
This follows from the short exact sequence
$$
0\longrightarrow N\cap N' \longrightarrow N\oplus N' \longrightarrow N+N' \longrightarrow 0
$$
and Lemma \ref{regseq}.

The main task is to show by an induction on $t$ that $v_k$ is
regular on each module of the form $v_{i_1}\cdots
v_{i_r}(v_{j_1},\ldots,v_{j_t})M$ for
$\{i_1,\ldots,i_r,j_1,\ldots,j_t\}\subseteq\{1,\ldots,s\}\setminus
\{k\}$ and all $k=1,\ldots,s$.

Then with $r=0$, $t=k-1$ this means that $v_k$ is
$(v_1,\ldots,v_{k-1})M$-regular, with $r=t=0$ that $v_k$ is
$M$-regular. Hence the exact sequence
$$0\longrightarrow (v_1,\ldots,v_{k-1})M\longrightarrow M \longrightarrow M/(v_1,\ldots,v_{k-1})M \longrightarrow 0$$
implies by Lemma \ref{regseq} that $v_k$ is $M/(v_1,\ldots,v_{k-1})M$-regular for all $k=1,\ldots,s$.

For the induction on $t$ let $t=0$. For simplicity we show that $v_k$ is
$v_1\cdots v_{k-1}M$-regular. But as permutations of regular
sequences are regular sequences, the proof works for arbitrary
elements of the sequence as well.

The Cartan homology of $\vb$ with values  in $M^*$ vanishes (see Proposition \ref{reg over cartan}) which
implies by Proposition \ref{co and hom} that the Cartan cohomology
of $\vb$ with values in $M$ vanishes. In particular
$$
0=\tilde{H}^0(k;M)=\frac{0:_{0:_M(v_1,\ldots,v_{k-1})}v_k}{v_k(0:_M(v_1,\ldots,v_{k-1}))}
$$
for all $k=1,\ldots,s$.
We show by a second induction on $k$ that
$v_k$ is $v_{k-1}\cdots v_{1}M$-regular.

If $k=1$ we have
$$
0=\tilde{H}^0(1;M)=\frac{0:_M v_1}{v_1M}.
$$
Hence $v_1$ is $M$-regular.
Now suppose that the assertion is known for $k-1$.

The module $0:_M(v_1,\ldots,v_{k-1})$ contains all elements of $M$
that are annihilated by all $v_i$, $i=1,\ldots,k-1$. Since every
$v_i$ is $M$-regular (where we use the same argument as for $v_1$, since permutations
of regular sequences are regular sequences),
$0:_M(v_1,\ldots,v_{k-1})=v_{k-1}M\cap\ldots\cap v_{1}M$. We show
$v_{l-1}M\cap\ldots\cap v_{1}M=v_{l-1}\cdots v_{1}M$ by another
induction on $l$, $2\leq l \leq k$.

If $l=2$ this is obvious. Now if $l>2$ we have
$$
v_{l-1}M\cap\ldots\cap v_{1}M=v_{l-1}M\cap(v_{l-2}M\cap\ldots\cap v_1M)
=v_{l-1}M\cap v_{l-2}\cdots v_1M\stackrel{(*)}{=}v_{l-1}\cdots v_{1}M
$$
where the induction hypothesis of the induction on $k$ is used,
i.e.\ that $v_{l-1}$ is regular on $v_{l-2}\cdots v_{1}M$ since $l-1\leq k-1$.

Then
$$
0=\tilde{H}^0(k;M)=\frac{0:_{0:_M(v_1,\ldots,v_{k-1})}v_k}{v_k(0:_M(v_1,\ldots,v_{k-1}))}=
\frac{0:_{v_{k-1}\cdots v_{1}M}v_k}{v_k(v_{k-1}\cdots v_{1}M)}
$$
implies that $v_k$ is $v_{k-1}\cdots v_{1}M$-regular.
Thus we proved the basis for the induction on $t$.

Now suppose $t>0$. We decompose $v_{i_1}\cdots
v_{i_r}(v_{j_1},\ldots,v_{j_t})M$ in two parts. By induction
hypothesis $v_k$ is regular on $v_{i_1}\cdots
v_{i_r}(v_{j_1},\ldots,v_{j_{t-1}})M$ and on $v_{i_1}\cdots
v_{i_r}v_{j_t}M$. Furthermore the induction hypothesis gives that
$v_{j_t}$ is $v_{i_1}\cdots
v_{i_r}(v_{j_1},\ldots,v_{j_{t-1}})M$-regular. Hence it follows from $(*)$
that the intersection of the two parts is
$$
v_{i_1}\cdots v_{i_r}(v_{j_1},\ldots,v_{j_{t-1}})M\cap v_{i_1}\cdots v_{i_r}v_{j_t}M=v_{j_t}v_{i_1}\cdots v_{i_r}(v_{j_1},\ldots,v_{j_{t-1}})M.
$$
Again by induction hypothesis $v_k$ is regular on this intersection.
So $(**)$ implies that $v_k$ is regular on $v_{i_1}\cdots
v_{i_r}(v_{j_1},\ldots,v_{j_{t-1}})M+v_{i_1}\cdots
v_{i_r}v_{j_t}M=v_{i_1}\cdots v_{i_r}(v_{j_1},\ldots,v_{j_t})M$.
This concludes the proof of our induction on $t$.

Finally it remains to show that $M/(v_1,\ldots,v_s)M\not=0$ for
$v_1,\ldots,v_s$ being an $M$-regular sequence. To this end we prove
by (a new) induction on $s$ that $(M/(v_1,\ldots,v_s)M)^*\cong v_s\cdots
v_1M^*$. If $s=1$ this follows from the exact sequence
$$
0\longrightarrow 0:_Mv_1 \longrightarrow M \stackrel{\cdot v_1}{\longrightarrow} M \longrightarrow M/v_1M \longrightarrow 0
$$
and the corresponding exact dual sequence
$$
0\longrightarrow (M/v_1M)^* \longrightarrow M^* \stackrel{\cdot v_1}{\longrightarrow} M^* \longrightarrow (0:_Mv_1)^* \longrightarrow 0
$$
because here $(M/v_1M)^*$ is the kernel of the multiplication with
$v_1$ which is $v_1M^*$ as $v_1$ is $M^*$-regular.

Now suppose $s>1$ and the assertion is proved for sequences of
length $<s$.

An induction on $r$ similar as in the first part of the proof shows
that $v_k$ is regular on $v_{i_1}\cdots
v_{i_r}(v_{j_1},\ldots,v_{j_t})M^*$ for
$\{i_1,\ldots,i_r,j_1,\ldots,j_t\}\subseteq [s]\setminus \{k\}$ for
$k=1,\ldots,s$, this time using the decomposition $$v_{i_1}\cdots
v_{i_r}(v_{j_1},\ldots,v_{j_t})M^*=v_{i_1}\cdots
v_{i_{r-1}}(v_{j_1},\ldots,v_{j_t})M^*\cap v_{i_1}\cdots
v_{i_{r-2}}v_{i_r}(v_{j_1},\ldots,v_{j_t})M^*.$$

In particular, $v_s$ is regular on $v_{s-1}\cdots v_1M^*$.
By the induction hypothesis (of the induction on $s$)
we have
$$
v_s\cdots v_1M^*=v_s(v_{s-1}\cdots v_1M^*)\cong v_s (M/(v_{s-1},\ldots, v_1)M)^*.
$$
We have already seen that $v_s$ is regular on $M/(v_{s-1},\ldots,
v_1)M$ and hence also on its dual $(M/(v_{s-1},\ldots, v_1)M)^*$.
Thus a second application of the induction hypothesis gives
\begin{eqnarray*}
v_s (M/(v_{s-1},\ldots, v_1)M)^*&\cong& \big(M/(v_{s-1},\ldots, v_1)M\big/v_s(M/(v_{s-1},\ldots, v_1)M)\big)^*\\
 &\cong& (M/(v_s,v_{s-1},\ldots,v_1)M)^*.
 \end{eqnarray*}

The module $M/(v_s,v_{s-1},\ldots,v_1)M$ is zero if and only if
$(M/(v_s,v_{s-1},\ldots,v_1)M)^*$ is zero. As we have just seen the
latter is isomorphic to $v_1\cdots v_sM^*$. If $v_1\cdots v_sM^*$
were zero, this would imply $0=v_1\cdots v_sM^*=0:_{v_1\cdots
v_{s-1}M^*}v_s=v_1\cdots v_{s-1}M^*$. Inductively we would obtain
$M^*=0$, a contradiction. This concludes the proof.
\end{proof}

We state a corollary which has been proved by the way in the proof
of Theorem \ref{regdual}.

\begin{cor}\label{dualmodreg}
Let $M\in\Mcc$ and $v_1,\ldots,v_s$ be an $M^*$-regular sequence.
Then
$$
(M/(v_1,\ldots, v_s)M)^*\cong v_1\cdots v_sM^*
$$
as graded $E$-modules.
\end{cor}

The relation between Cartan homology and Cartan cohomology in
Proposition \ref{co and hom} provides the following corollary.

\begin{cor}\label{reg over hom or cohom}
Let $M\in\Mcc$ and $\vb= v_1,\ldots,v_s\in E_1$. Then the following
statements are equivalent:
\begin{enumerate}
 \item $v_1,\ldots,v_s$ is $M$-regular;
 \item $H_1(\vb;M)=0$;
 \item $H_i(\vb;M)=0$ for all $i>0$;
 \item $H^1(\vb;M)=0$;
 \item $H^i(\vb;M)=0$ for all $i>0$.
\end{enumerate}
\end{cor}
\begin{proof}
The equivalence of the first three conditions is stated in
Proposition \ref{reg over cartan}. An $E$-module is zero if and only
if its dual is zero. Thus the equality of condition (ii) and (iv)
resp. (iii) and (v) follows from $H_i(\vb;M^*)\cong H^i(\vb;M)^*$ as seen
in Proposition \ref{co and hom}.
\end{proof}

\section{Modules with linear injective resolutions}
\label{sect linear injective}

In this section we focus on $E$-modules having linear injective
resolutions. We begin with an example.

\begin{ex}\label{ex face ring}
Let $\Delta$ be a simplicial complex on $[n]$.
Then $\Delta$ is Cohen-Macaulay if and only if
the face ideal $J_{\Delta^*}=(e_F: F\not\in \Delta^*)$ of the
Alexander dual $\Delta^*=\{F\subseteq [n]: F^c\not \in \Delta\}$
(here $F^c$ denotes the complement of $F$ in $[n]$) has a linear
projective resolution as was shown in \cite[Corollary 7.6]{AH}.

This is equivalent to say that the face ring $K\{\Delta\}=E/J_\Delta$ has a linear injective
resolution as it is the dual $(J_{\Delta^*})^*\cong
E/(E/J_{\Delta^*})^*\cong E/0:_EJ_{\Delta^*}\cong E/J_\Delta$
of $J_{\Delta^*}$.
\end{ex}

If $v\in E_1$ is $M$-regular then $M$ has a $t$-linear projective
resolution over $E$ if and only if $M/vM$ has a $t$-linear
resolution over $E/(v)$.  Linear injective
resolutions behave more complicated under reduction modulo regular
elements.

\begin{lem}\label{lin inj mod reg}
Let $M\in\Mcc$ and $v\in E_1$ be an $M$-regular element. Then $M$ has a
$d$-linear injective resolution over $E$ if and only if $vM$ has a
$d$-linear injective resolution over $E/(v)$.

In particular, if $v$ is $E/J$-regular
for some graded ideal $J\subset E$, then we have that $E/J$ has a $d$-linear
injective resolution over $E$ if and only if $E/(J+(v))$ has a
$(d-1)$-linear injective resolution over $E/(v)$.
\end{lem}
\begin{proof}
Let
$$
I^{\pnt}:\qquad 0\longrightarrow \Dirsum_{j\in\Z}E(n-j)^{\mu_{0,j}(M)}\longrightarrow \Dirsum_{j\in\Z}E(n-j)^{\mu_{1,j}(M)} \longrightarrow \ldots
$$
be the minimal graded injective resolution of $M$ over $E$. We claim that
$\Hom_E(E/(v),I^{\pnt})$ is the minimal graded injective resolution of
$\Hom_E(E/(v),M)\cong 0:_Mv=vM$ over $E/(v)$ with the same ranks and
degree shifts, i.e. $\mu^{E/(v)}_{i,j}(vM)=\mu^E_{i,j}(M)$. From
this the claim follows.

The homology of $\Hom_E(E/(v),I^{\pnt})$ is isomorphic to
the Cartan
cohomology $H^i(v;M)$ of $M$ with respect to $v$ by Proposition
\ref{TorExt and Cartan}. As $v$ is $M$-regular, Corollary \ref{reg
over hom or cohom} implies that $H^i(v;M)=0$ for $i>0$. So this is
indeed a resolution of $\Hom_E(E/(v),M)\cong vM$.

The modules in this resolution are
\begin{eqnarray*}
\Hom_E(E/(v),\Dirsum_{j\in\Z} E(n-j)^{\mu_{i,j}(M)})&\cong& \Dirsum_{j\in\Z}\Hom_E(E/(v),E)(n-j)^ {\mu_{i,j}(M)}\\
 &\cong& \Dirsum_{j\in\Z}(0:_E v)(n-j)^ {\mu_{i,j}(M)}\\
 &\cong& \Dirsum_{j\in\Z}(v)(n-j)^ {\mu_{i,j}(M)}\\
 &\cong& \Dirsum_{j\in\Z}E/(v)(n-1-j)^ {\mu_{i,j}(M)}.\\
\end{eqnarray*}
Thus $\Hom_E(E/(v),I^{\pnt})$ is an injective resolution of $vM$
with $\mu^{E/(v)}_{i,j}(vM)=\mu^E_{i,j}(M)$ (bear in mind that
$E/(v)$ is an exterior algebra with $n-1$ variables). The minimality
is preserved because an injective resolution over $E$ is minimal if and only
if all entries in the matrices of the maps are in the maximal ideal.
This property is not touched by applying $\Hom_E(E/(v),\cdot)$.

Now suppose $M=E/J$ for some graded ideal $J$. As just proved $E/J$
has a $d$-linear injective resolution over $E$ if and only if $v(E/J)$ has
one over $E/(v)$. The latter module is isomorphic to the $E/(v)$-module
$$v(E/J)=\big(J+(v)\big)/J\cong \big(E/(v)\big)\big/\big(J+(v)/(v)\big)(-1)$$
where the isomorphism is induced by the homomorphism $$E/(v)\to
\big(J+(v)/J\big)(+1), \ a+(v)\mapsto av+J.$$ Thus $v(E/J)$ has a
$d$-linear injective resolution if and only if
$\big(E/(v)\big)\big/\big(J+(v)/(v)\big)$ has a $(d-1)$-linear
injective resolution.
\end{proof}

Recall that for $0\neq M\in\Mcc$ the number
$$
\reg M =\max\{j-i: \Tor^E_i(M,K)_{j}\not=0\}
$$
is the regularity of $M$.
The regularity of $M$ is bounded by $d(M)=\max\{i:M_i\not=0\}$ which
can be seen when computing $\Tor^E_i(M,K)$ via the Cartan complex.
For a graded ideal $0\neq J\subset E$ there
is the relationship $\reg J=\reg E/J+1$. Reducing modulo a regular
element $v$ does not change the regularity because the minimal
graded free resolution of $M/vM$ over $E/(v)$
has the same ranks and shifts as the minimal graded free
resolution of $M$ over $E$.
For quotient rings with linear injective resolution there is a nice
formula for the regularity.

\begin{thm}\label{reg+depth=d}
Let $|K|=\infty$ and $E/J$ have a $d$-linear injective resolution.
Then
$$
\reg  E/J+\depth  E/J=d.
$$
\end{thm}
\begin{proof}
At first assume $\depth  E/J=0$. Then $d=d(E/J)$ is an upper bound
for $\reg E/J $ and we want to show that both numbers are equal. In
\cite[Theorem 5.3]{AH} it is proved that $\reg E/J=\reg E/\gin(J)$.
Thus we may assume in addition that $J$ is strongly stable. Then by
\cite[Corollary 3.2]{AHH} the regularity of $J$ is
$$\reg J=\max\{\deg(u):u\in G(J)\}.$$
In particular, $J$ is  a monomial ideal such that it can be seen as
the face ideal of a simplicial complex $\Delta$, i.e.
$J=J_\Delta=(e_F: F\not\in\Delta)$. Then we have already seen that
$(E/J)^*\cong 0:_E J=J_{\Delta^*} $ is generated by all monomials $e_F$ with
$e_{F^c}\not\in J$ (cf. Example \ref{ex face ring}).

From $n=\cx  E/J+\depth  E/J=\cx  E/J=\cx J$ and Proposition \ref{cx}
follows the existence of a monomial $e_Fe_n\in G(J)$. We have
$e_Fe_i\in J$ for all $i=1,\ldots,n$ because $J$ is stable. But then
$e_{(F\cup\{i\})^c}=e_{F^c\setminus\{i\}}\not\in (E/J)^*$ for all
$i\not\in F$. As $e_F\not\in J$ (otherwise $e_Fe_n$ would not be a
minimal generator), the complement $e_{F^c}$ is in $(E/J)^*$ and
even a minimal generator. The ideal $(E/J)^*$ has an $(n-d)$-linear
projective resolution and is thus generated in degree $n-d$, so
$|F|=n-|F^c|=n-(n-d)=d.$

This means that there exists a minimal generator of $J$ of degree
$d+1$ which implies $\reg E/J=\reg J-1=d+1-1=d$.

Now suppose $\depth  E/J=s$. Reducing modulo a maximal regular
sequence $v_1,\ldots,v_s$ does not change the regularity, but
$E/J+(v_1,\ldots,v_s)$ has a $(d-s)$-linear injective resolution
over $E/(v_1,\ldots,v_s)$ by Lemma \ref{lin inj mod reg}. Then
$\reg(E/J+(v_1,\ldots,v_s))=d-s$ and so
$$\reg E/J=\reg(E/J+(v_1,\ldots,v_s))=d-s=d-\depth  E/J.$$
\end{proof}

\begin{rem}
Let $|K|=\infty$ and $0\neq J\subset E$ with $d$-linear injective
resolution. By \cite[Theorem 3.2]{AAH}
we have $\cx E/J=n-\depth E/J.$
As $d\leq n$ this proves that
$$\reg E/J\leq \cx E/J.$$
This
inequality is even true for general quotient rings $E/J$. For
arbitrary graded $E$-modules there is no such relation between the
regularity and the complexity since the first one is changed by
shifting while the other is invariant.
\end{rem}

For a graded ideal $J\subset E$  Eisenbud, Popescu and Yuzvinsky
characterize in \cite{EPY} the case when both $J$ has a linear
projective and $E/J$ a linear injective resolution over $E$. In
their proof they use the Bernstein-Gel'fand-Gel'fand-correspondence
between resolutions over $E$ and resolutions over the polynomial
ring in $n$ variables. We present a (partly) more direct
proof using generic initial ideals.

\begin{thm}\label{linear resolution}\cite[Theorem 3.4]{EPY}
Let $|K|=\infty$ and $0\neq J\subset E$ be a graded ideal. Then $J$ and
$(E/J)^*$ have linear projective resolutions if and only if $J$
reduces to a power of the maximal ideal modulo some (respectively
any) maximal $E/J$-regular sequence of linear forms of $E$.
\end{thm}
\begin{proof}
At first we show that it is enough to consider the case $\depth
E/J=0$. Note that the
ideal $J$ has a linear projective resolution over $E$ if and only if
$J+(v_1,\ldots,v_s)/(v_1,\ldots,v_s)$ has a linear projective resolution over
$E/(v_1,\ldots,v_s)$. Furthermore Lemma \ref{lin inj mod reg} says
that $E/J$ has a linear injective resolution over $E$ if and only if
the $E/(v)$-module $E/J+(v)$ has a linear injective resolution for
some $E/J$-regular element $v$. Thus inductively $E/J$ has a linear
injective resolution over $E$ if and only if
$E/(J+(v_1,\ldots,v_s))$ has one over $E/(v_1,\ldots,v_s)$. All in
all we may indeed assume that $\depth E/J=0$.

The $t$-th power of the maximal ideal $\mm=(e_1,\ldots,e_n)$ has a
$t$-linear projective resolution because it is strongly stable and
generated in one degree (cf. Lemma \ref{betti}). For the same reason
$(E/\mm^t)^*\cong 0:_E\mm^t=\mm^{n-t+1}$ has a linear projective
resolution. Hence the ``if'' direction is proved.

Now it remains to show that if $J$ has a $t$-linear projective resolution,
$E/J$ has a $d$-linear injective resolution and $\depth  E/J=0$,
then $J=\mm^t$.

In a first step we will see that $J$ may be replaced by its generic
initial ideal. If $J$ has a $t$-linear projective resolution, its
regularity is obviously $t$. Then by \cite[Theorem 5.3]{AH} the
regularity of $\gin(J)$ is also $t$. As $\gin(J)$ is generated in
degree $\geq t$ this implies that $\gin(J)$ has a $t$-linear
resolution as well.

Generic initial ideals and duality commute by Lemma \ref{dual gin},
i.e.
$$\gin((E/J)^*)\cong (E/\gin(J))^*.$$
Then a similar argument shows that $E/\gin(J)$ has a $d$-linear
injective resolution as well.
Finally
$$\depth  E/\gin(J)=\depth  E/J=0$$
by \cite[Proposition 2.3]{HT}. Altogether $\gin(J)$ satisfies the same
conditions as $J$. Assume that $\gin(J)=\mm^t$. The Hilbert
series of $J$ and $\gin(J)$ are the same which implies that in
this case $J=\mm^t$ as well because $J\subseteq \mm^t=\gin(J)$.

This allows us to replace $J$ by $\gin(J)$ so in the following we
assume that $J$ is strongly stable.

In the proof of Theorem \ref{reg+depth=d} was proved in the same
situation that there exists a minimal generator of $J$ of degree
$d+1$. As $J$ is generated in degree $t$, this implies $d=t-1$.

Finally, we will see that this equality implies $J=\mm^t$. As $E/J$
has a $d$-linear injective resolution, the number
$d(E/J)=\max\{i:(E/J)_i\not=0\}$ equals $d$.
Then, by Proposition \ref{d(EJ)},
$$\max\{\min(u):u\in G(J)\}=n-d=n-t+1.$$
Thus there exists a monomial $u\in G(J)$ of degree $t$ with
$\min(u)=n-t+1$. The only possibility for $u$ is
$u=e_{n-t+1}\cdots e_n$. Then every monomial of degree $t$ is in $J$
because $J$ is strongly stable and this implies $J=\mm^t$ since $J$
is generated in degree $t$.
\end{proof}

Let $v\in E_1$.
Recall that  $H_i(M,v)$ is the homology
of the complex
$$
(M,v): \qquad \ldots\longrightarrow M_{i-1}\stackrel{\cdot v }{\longrightarrow} M_{i}\stackrel{\cdot v }{\longrightarrow} M_{i+1}\longrightarrow \ldots
$$
In Section \ref{sect os algebras} we need the following technical result from  \cite{EPY}.

\begin{thm}\label{d enough}\cite[Theorem 4.1(b)]{EPY}
Let $M\in\Mcc$ have a $d$-linear injective resolution. Then
$H_i(M,v)=0$ for all $i\in\Z$ if and only if $H_d(M,v)=0$.
\end{thm}

\section{Orlik-Solomon algebras}
\label{sect os algebras}

In this section we investigate homological properties of the
Orlik-Solomon algebra of a matroid. It is one example for
$E$-modules with linear injective resolutions. We determine the
depth and the regularity of the Orlik-Solomon algebra and
characterize the matroids whose Orlik-Solomon ideal has a linear
resolution. In the following the letter ``$M$'' denotes always a matroid
and never a module.

For the convenience of the reader we first collect all necessary
matroid notions that will be used in this section. They can be found
in introductory books on matroids, as for example \cite{Ox} or
\cite{We}.

Let $M$ be a non-empty matroid over $[n]=\{1,\ldots,n\}$, i.e. $M$
is a collection $\Ic$ of subsets of $[n]$, called \emph{independent
sets}, satisfying the following conditions:
\begin{enumerate}
 \item $\emptyset \in \Ic$.
 \item If $A\in\Ic$ and $B\subseteq A$, then $B\in\Ic$.
 \item If $A,B\in\Ic$ and $|A|<|B|$, then there exists an element $i\in B\setminus A$ such that $A\cup \{i\}\in\Ic$.
\end{enumerate}
The subsets of $[n]$ that are not in $\Ic$ are called
\emph{dependent}, minimal dependent sets are called \emph{circuits}.
The cardinality of maximal independent sets (called \emph{bases}) is
constant and denoted by $r(M)$, the \emph{rank} of $M$.

On $E$ exists a derivation $\partial : E\rightarrow E$ of degree
$-1$ which maps $e_i$ to 1 and obeys the Leibniz rule
$$
\partial(ab)=(\partial a)b+(-1)^{\deg a}a(\partial b)
$$
 for homogeneous $a\in E$ and all $b\in E$.  One easily checks
$$
\partial  e_S=(e_{i_1}-e_{i_0})\cdots(e_{i_m}-e_{i_0})=\sum_{j=0}^m(-1)^{j}e_{S\setminus \{i_j\}}
$$
for $S=\{i_0,\ldots,i_m\}$.
The \emph{Orlik-Solomon ideal} of $M$ is the ideal
$$J(M)=(\partial e_S:S\text{ is dependent})=(\partial e_C:C\text{ is a circuit}).$$
If there is no danger of confusion we simply write $J$ for $J(M)$.
The quotient ring $E/J$ is called the \emph{Orlik-Solomon algebra}
of $M$.

A circuit whose minimal element (with respect to a chosen order on
$[n]$) is deleted is called a \emph{broken circuit}. A set that does
not contain any broken circuit is called \emph{nbc}. Bj\"orner
proves in \cite[Theorem 7.10.2]{Bj} that the set of all $\nbc$-sets
is a $K$-linear basis of $E/J$.

A \emph{loop} is a subset $\{i\}$ that is dependent. If $M$ has a
loop $\{i\}$, then $\partial e_i=1$ is in $J$ and thus $E/J$ is
zero. Quite often it is enough to consider the case that $M$ is
\emph{simple}, i.e. $M$ has no loops and no non-trivial parallel
classes. A \emph{parallel class} is a maximal subset such that any
two distinct members $i,j$ are parallel, i.e. $\{i,j\}$ is a
circuit.

Note that if $M$ has no loops, a monomial $e_S$ is contained in $J$
if and only if the set $S$ is dependent (see for example \cite[Lemma
7.10.1]{Bj}).

\begin{ex}\label{ex umn}
The simplest matroids are the \emph{uniform matroids} $U_{m,n}$ with
$m\leq n$. They are matroids on $[n]$ such that all subsets of
$[n]$ of cardinality $\leq m$ are independent. The rank of $U_{m,n}$
is obviously $m$ and the circuits of $U_{m,n}$ are all subsets of
$[n]$ of cardinality $m+1$. Thus the Orlik-Solomon ideal
$J_{m,n}:=J(U_{m,n})$ of $U_{m,n}$ is the ideal $J_{m,n}=(\partial
e_A: A\subset[n], |A|=m+1)$. The relation
$$\partial e_S=\sum_{j=0}^k(-1)^j\partial e_{S\setminus\{i_j\}\cup\{1\}}$$
for $S=\{i_0,\ldots,i_k\}\subseteq [n]$ with $1\not\in S$ is easily
verified by a simple computation. Then we can rewrite the
Orlik-Solomon ideal as
$$J_{m,n}=(\partial e_A:A\subset[n], |A|=m+1, 1\in A).$$
\end{ex}

The rank of a subset $X\subseteq[n]$ is the rank of the matroid
$M|X$ which results from restricting $M$ on $X$. Then the closure
operator $\cl$ is defined as
$$
\cl(X)=\{i\in [n]:r(X\cup\{i\})=r(X)\}
$$
for $X\subseteq[n]$. If $\cl(X)=X$, then $X$ is called a \emph{flat}
(or a closed set). The by inclusion partially ordered set $L$ of all
flats of $M$ is a graded lattice. On $L$ we consider the
\emph{M\"obius function} which can be defined recursively by
$$
\mu(X,X)=1 \qquad \text{ and } \qquad \mu(X,Z)=-\sum_{X\leq Y<Z}\mu(X,Y) \text{ if } X<Z
$$
and the \emph{characteristic polynomial}
$$p(L;t)=\sum_{X\in L}\mu(\emptyset,X)t^{r(M)-r(X)}.$$

The \emph{beta-invariant} $\beta(M)$ of a matroid $M$ was introduced
by Crapo in \cite{Cr} as
$$
\beta(M)=(-1)^{r(M)}\sum_{S\subseteq [n]}(-1)^{|S|}r(S)=(-1)^{r(M)}\sum_{X\in L}\mu(\emptyset,X)r(X).
$$

The M\"obius function, the characteristic polynomial and the
beta-invariant are considered in detail, e.g., in \cite{Za}.

The \emph{direct sum} of two matroids $M_1$ and $M_2$ on disjoint
ground sets $E_1$ and $E_2$ is the matroid $M_1\oplus M_2$ on the
ground set $E_1\cup E_2$ whose independent sets are the unions of an
independent set of $M_1$ and an independent set of $M_2$. The
circuits of $M_1\oplus M_2$ are those of $M_1$ and those of $M_2$.
The Hilbert series of the Orlik-Solomon algebra is multiplicative on
direct sums, i.e.
$$H(E/J(M_1\oplus M_2),t)=H(E/J(M_1),t)\cdot H(E/J(M_2),t).$$
This can be proved using the fact that the set of all $\nbc$-sets of
cardinality $k$ is a $K$-basis of $(E/J)_k$ and that the $\nbc$-sets
of $M_1\oplus M_2$ are the unions of an $\nbc$-set of $M_1$ and an
$\nbc$-set of $M_2$.

On a matroid $M$ exists the equivalence relation
$$
x \sim y \Leftrightarrow x=y \text{ or there is a circuit  which contains both } x \text{ and } y .
$$

The equivalence classes of this relation are called the
\emph{connected components} or, more briefly, \emph{components} of
$M$. They are disjoint subsets of the ground set and each circuit
contains only elements of one component. If $T_1,\ldots,T_k$ are the
components of $M$ then $M=M|T_1\oplus\dots\oplus M|T_k$. The matroid
$M$ is called \emph{connected} if it has only one connected
component.

The Orlik-Solomon algebra has a linear injective resolution, which
was first observed by Eisenbud, Popescu and Yuzvinsky in \cite{EPY} for
Orlik-Solomon algebras defined by hyperplane arrangements, although
their proof works for arbitrary Orlik-Solomon algebras as well.
For the convenience of the reader we present a compact proof.
\begin{thm}\label{linear inj}\cite[Theorem 1.1]{EPY}
Let $l=r(M)$ be the rank of the matroid $M$. Then the Orlik-Solomon
algebra $E/J$ of $M$ has an $l$-linear injective resolution.
\end{thm}
\begin{proof}
Let $\Gamma$ be the simplicial complex whose faces are the nbc-sets
of $M$.  The face ideal of $\Gamma$ is the ideal
$$
J_\Gamma=(e_A:A\not\in\Gamma)=(e_A: A \text{ is a broken circuit}).
$$
This ideal is just the initial ideal $\ini( J)$ of $J$ (this is
implicitly contained in the proof of \cite[Theorem 3.3]{DY}).

By \cite[Theorem 7.4.3]{Bj} the complex $\Gamma$ is shellable and
hence Cohen-Macaulay. So as in Example \ref{ex face ring} it follows
that $E/J_\Gamma=E/\ini(J)$ has a linear injective resolution.
Then Corollary \ref{mu gin} implies that $E/J$ has a linear
injective resolution, too.

Every subset of $[n]$ of cardinality
greater than $l$ is dependent and thus every monomial of degree
greater than $l$ is contained in $J$.  Hence
$d(E/J)=\max\{i:(E/J)_i\not=0\}\leq l$. There exists an independent
subset $A\subseteq [n]$ of cardinality $l$. Then $e_A\not\in J$ and
$d(E/J)=l$. So $E/J$ has an $l$-linear injective resolution as was observed in
Section \ref{sect Cartan complex}.
\end{proof}

Next we want to determine the depth of the Orlik-Solomon
algebra. We are able to find at least one $E/J$-regular element if
$M$ has no loops.

\begin{prop}\label{ei regular}
If the matroid $M$ has no loops, then the variable $e_i$ is
$E/J$-regular for all $i\in [n]$. In particular, $\depth  E/J\geq
1$.
\end{prop}
\begin{proof}
By Theorem \ref{d enough} and  Theorem \ref{linear inj} it is enough to
show that the annihilator of $e_i$ in $E/J$ and the ideal
$(\bar{e_i})=e_i(E/J)$ in $E/J$ coincide in degree $l$.

Every set of cardinality $l+1$ is dependent and therefore every
monomial of degree $l+1$ is contained in $J$ whence $(E/J)_{l+1}=0$.
So every element in $E/J$ of degree $l$ is annihilated by $e_i$.

Now
let $T$ be an independent set of cardinality $l$ that does not
contain $i$. Then $T\cup\{i\}$ is dependent and thus $\partial
e_{T\cup\{i\}}\in J$. Arrange $T\cup\{i\}$ such that $i$ is the
first element. Then in $E/J$ there is the relation
$$
\overline{e_T}=\overline{e_T}-\overline{\partial e_{T\cup\{i\}}}=\overline{e_T}-\overline{e_T+(\ldots)e_i}=(\ldots)\overline{e_i}.
$$
So the residue class of every monomial of degree $l$ is in the ideal
generated by $\overline{e_i}$, which shows that the annihilator and
the ideal $(\overline{e_i})$ coincide in degree $l$. This shows that
$e_i$ is $E/J$-regular and thus the depth of $E/J$ is at least 1.
\end{proof}

The matroids $M$ whose corresponding depth is exactly 1 can be
characterized by their \emph{beta-invariant} $\beta(M)$.

\begin{thm}\label{depth}
If $|K|=\infty$ and $M$ has no loops, then the depth of the Orlik-Solomon
algebra $E/J$ equals 1 if and only if $\beta(M)\not=0$.
\end{thm}
\begin{proof}
Theorem \ref{hilbert series} shows that the depth of $E/J$ is the
maximal number $s$ such that the Hil\-bert se\-ries can be written
as $H(E/J,t)=(1+t)^sQ(t)$ for some $Q(t)\in\Z[t]$.

Bj\"orner proves in \cite[Corollary 7.10.3]{Bj} that
$$H(E/J,t)=(-t)^{r(M)}p(L;-\frac{1}{t}).$$
Replacing the characteristic polynomial $p(L;-\frac{1}{t})$ by its definition gives
$$H(E/J,t)=\sum_{X\in L}\mu(\emptyset,X)(-1)^{r(X)}t^{r(X)}.$$
Thus the Taylor expansion of $H(E/J,t)$ at $-1$ is
\begin{eqnarray*}
 H(E/J,t)& = & \sum_{X\in L}\mu(\emptyset,X)(-1)^{r(X)}r(X)(-1)^{r(X)-1}(1+t)+(1+t)^2(\ldots)\\
         & = & -\sum_{X\in L}\mu(\emptyset,X)r(X)(1+t)+(1+t)^2(\ldots)\\
         & = & (-1)^{r(M)-1}\beta(M)(1+t)+(1+t)^2(\ldots).\\
\end{eqnarray*}
Now one sees that $H(E/J,t)$ can be divided twice by $1+t$ if and
only if $\beta(M)=0$. Observe that $H(E/J,-1)=0$ because $1+t$
divides $H(E/J,t)$ at least once since $e_i$ is regular on $E/J$ by
the preceding lemma.
\end{proof}

Crapo \cite[Theorem II]{Cr} proved that $M$ is connected if and only
if $\beta(M)\not=0$ (see also Welsh \cite[Chapter 5.2]{We}). Thus
the above result says that if $M$ is connected, the depth of $E/J$
equals the number of components of $M$. This is true in general.

\begin{thm}\label{depth general}
Let $|K|=\infty$ and $M$ be a loopless matroid with $k$ components
and $J$ its Orlik-Solomon ideal. Then $\depth  E/J=k$.
\end{thm}
\begin{proof}
Let $M_1,\ldots,M_k$ be the matroids on the components of $M$, i.e.
$M=M_1\oplus\ldots\oplus M_k$ and let $J_i=J(M_i)$ be the
corresponding Orlik-Solomon ideals. Theorem \ref{hilbert series} and
Theorem \ref{depth} imply that their Hilbert series can be
written as
$$H(E/J_i,t)=Q_i(t)\cdot (1+t)$$
such that $Q_i(-1)\not=0$. The Hilbert series is multiplicative on direct sums, thus
$$H(E/J,t)=\prod_{i=1}^k \left( Q_i(t)\cdot(1+t)\right)=Q(t)\cdot (1+t)^k$$
with $Q(-1)\not=0$ and so $\depth E/J=k$.
\end{proof}

For Orlik-Solomon algebras of hyperplane arrangements this result
was already proved by Eisenbud, Popescu and Yuzvinsky. In \cite[Corollary
2.3]{EPY} they state that the codimension of the singular variety
(i.e. the set of all non-regular elements on the Orlik-Solomon
algebra) of the arrangement is the number of central factors in an
irreducible decomposition of the arrangement. This codimension is
exactly the depth of the Orlik-Solomon algebra as
Aramova, Avramov and Herzog showed in \cite[Theorem 3.1]{AAH}.

\begin{rem}\label{max reg seq}
Let $M$ be a loopless matroid with components $T_1,\ldots,T_k$ and
$M_i=M|T_i$. A ``canonical'' maximal regular sequence on $E/J$ can
be found as follows. For every component $T_j$  choose an element
$i_j\in T_j$. Then $e_{i_1},\ldots,e_{i_k}$ is an $E/J$-regular
sequence. As $E/J+(e_{i_1},\ldots,e_{i_{j-1}})$ has an
$(l-j+1)$-linear injective resolution over
$E/(e_{i_1},\ldots,e_{i_{j-1}})$ by Lemma \ref{lin inj mod reg}, it
is enough to prove that $e_{i_j}$ is regular on
$E/J+(e_{i_1},\ldots,e_{i_{j-1}})$ in degree $l-j+1$. Let $A$ be an
independent subset of $[n]\setminus\{i_1,\ldots,i_{j-1}\}$ with
$|A|=l-j+1$. Then $A=S_1\cup\ldots\cup S_k$ with $S_i\subseteq T_i$.
The rank of $M$ is the sum of the ranks of the $M_i$, i.e.
$l=r(M_1)+\ldots+r(M_k)$. So at most $j-1$ of the $S_i$ are not
bases of their matroid, which means that there exists a
$t\in\{1,\ldots,j\}$ such that $S_t\cup\{i_t\}$ is dependent in
$M_t$. Then $A\cup\{i_t\}$ is dependent in $M$. The same trick as in
the proof of Proposition \ref{ei regular} shows that $e_A\in
J+(e_{i_1},\ldots,e_{i_j})$.
\end{rem}

As we know now the depth, we can compute the regularity of the
Orlik-Solomon algebra as well.
\begin{cor}\label{reg=l-k}
Let $|K|=\infty$ and $M$ be a loopless matroid of rank $l$ with $k$ components. The regularity of its Orlik-Solomon algebra is
$$\reg E/J=l-k.$$
\end{cor}
\begin{proof}
This is just an application of Theorem \ref{reg+depth=d}.
\end{proof}

\begin{ex}\label{ex umn 2}
We consider the uniform matroids $U_{m,n}$ and their Orlik-Solomon
ideals $J_{m,n}$.

If $m=0$ then every set is dependent. The circuits are all sets with
one element, in particular they are loops. Thus $U_{0,n}$ has rank 0
and $n$ components $U_{0,1}$. The Orlik-Solomon ideal is $J_{0,n}=E$.

If $m=n$ then every set is independent. There are no circuits hence
$J_{n,n}=0$. The rank of $U_{n,n}$ is $n$ and it has $n$ components
$U_{1,1}$. Thus $\depth  E/J=n$ and $\cx  E/J=0$. The regularity
is $\reg E/J=n-n=0$.

If $m\not=0,n$ then $U_{m,n}$ is connected. Thus $\depth  E/J=1$
and $\cx  E/J=n-1$. The rank is $m$ hence the regularity is
$\reg E/J=m-1$.
\end{ex}

We say that an $E$-module has \emph{linear relations} if it is
generated in one degree and the first syzygy module is generated in
degree one. Thus a linear projective resolution implies linear
relations.

\begin{thm}\label{linrelations}
Let $M$ be a simple matroid and have no singleton components. If the
Orlik-Solomon ideal $J$ has linear relations then $M$ is connected.
\end{thm}
\begin{proof}
As $M$ is simple there exists no circuits with one or two elements,
so  $J$ is generated in degree  $m\geq 2$. Suppose $J=(\partial
e_{C_i}:i=1,\ldots,r)$ where $C_1,\ldots,C_r$ are circuits of $M$ of
cardinality $m+1$. Let $f_1,\ldots,f_r$ be the free generators of
$\Dirsum_{i=1}^rE(-m)$ such that $f_i$ is mapped to $\partial
e_{C_i}$ in the minimal graded free resolution of $J$. Then the
assumption says that the kernel of this map,
$$
U=\{\sum_{i=1}^ra_if_i:a_i\in E,\sum_{i=1}^ra_i\partial e_{C_i}=0\},
$$
is generated by elements $r_k=\sum_{i=1}^r v_{ik}f_i$ with
$v_{ik}\in E_1$. We may assume that the generators $r_k$ are
minimal, i.e. no sum $\sum_{i\in I'}v_{ik}f_i$ with
$I'\subsetneq\{1,\ldots,r\}$ is  in $U$. The
support of a linear form $v=\sum_{j=1}^n\alpha_je_j$ with
$\alpha_j\in K$ is the set $\supp(v)=\{j:\alpha_j\not=0\}$.

Under this conditions we claim that for each $k$ the elements of the circuits $C_i$
with $v_{ik}\not=0$ are in the same component of $M$, which we call the component of $C_i$,
and
consequently the support of $v_{ik}$ is in this component, too.

The monomials in $\sum_{i=1}^r v_{ik}\partial e_{C_i}=0$ have the
form $e_je_{C_i\setminus\{l\}}$ with $l\in C_i$ and
$j\in\supp(v_{ik})$. Because of the structure of $\partial e_{C_i}$
the monomials $e_je_{C_i\setminus\{l\}}$ cannot be zero for all
$l\in C_i$. If it is not zero, then there exists $C_p$, $q\in C_p$
and $t\in\supp(v_{pk})$ such that
$$
\{j\}\cup C_i\setminus\{l\}=\{t\}\cup C_p\setminus\{q\}.
$$
As $C_i$ and $C_p$ have at least three elements, it follows that
their intersection is not empty. This means that their elements
are both in the same component of
$M$. Then the minimality of $r_k$ implies that all elements of circuits $C_i$
with $v_{ik}\not=0$ belong to the same component.

Every $j\in\supp(v_{ik})$ must belong to some circuit $C_p$ with
$v_{pk}\not=0$, otherwise we see that
$$
e_j\sum_{\{i:j\in\supp(v_{ik})\}}\alpha_{ijk}\partial e_{C_i}=0
$$
when $v_{ik}=\sum_{j=1}^n\alpha_{ijk}e_j$.
This implies
$\sum_{\{i:j\in\supp(v_{ik})\}}\alpha_{ijk}\partial e_{C_i}\in(e_j)$
and thus this sum equals zero. But this is not possible
by our assumption on $U$. Hence all indices of the support of
$v_{ik}$ belong to the same component of $M$ as
the elements of the circuits $C_i$.

If $M$ is not connected and has no singleton components, there
exists at least two components and thus two circuits $C_i$ and $C_j$
whose intersection is empty. There is a trivial relation of degree
$2m$ between the generators corresponding to these two circuits,
namely $\partial e_{C_l}f_j\pm\partial e_{C_j}f_l.$ This relation
has a representation
$$
\partial e_{C_l}f_j\pm\partial e_{C_j}f_l=\sum_kg_kr_k=\sum_k\sum_{i=1}^rg_kv_{ik}f_i
$$
where $g_k\in E_{m-1}$. Then
$$
\partial e_{C_l}=\sum_kg_kv_{jk}
$$
since the $f_i$ are free generators. Each monomial in the sum on the
right side has a variable whose index is in the support of $v_{jk}$.
As shown above this support is contained in the component of $C_j$.
Thus $C_l$ contains elements of the component of $C_j$ which implies
that both circuits belong to the same component, a contradiction to
the choice of $C_l$ and $C_j$.
\end{proof}

Finally we classify all Orlik-Solomon ideals with linear projective
resolutions. Only joining or removing ``superfluous'' variables has
no effect on the linearity of $J$. This operation can be expressed using the
direct sum of matroids. A singleton  $\{i\}$ is a
component of a (loopless) matroid $M$ if and only if it is contained
in no circuit, or equivalently, is contained in each base. In this
case $i$ is called a \emph{coloop}. The matroid on $\{i\}$ is
$U_{1,1}$ if $i$ is a coloop, so we can write $M=M'\oplus U_{1,1}$
with $M'=M|_{[n]\setminus\{i\}}$. Let $E'=E/(e_i)$ and $J'$ the Orlik-Solomon ideal
of $M'$ in $E'$. Then $e_i$ is $E/J$-regular and  $J=J'E$ has a
linear resolution if and only if $J'$ has one. By iterating this
procedure we can split up $M$ in the direct sum of a matroid $M'$
which has no singleton components and a copy of $U_{f,f}$ where $f$
is the number of coloops of $M$ (note that $U_{f-1,f-1}\oplus
U_{1,1}=U_{f,f}$). Then $J(M')$ has a linear resolution if and only
if $J(M)$ has one.

\begin{thm}\label{characterisation}
Let $|K|=\infty$ and $M$ be a matroid on $[n]$. The Orlik-Solomon
ideal $J$ of $M$ has an $m$-linear projective resolution if and only
if $M$ satisfies one of the following three conditions:
\begin{enumerate}
 \item $M$ has a loop and $m=0$.
 \item $M$ has no loops, but non-trivial parallel classes, $m=1$ and $M=U_{1,n_1}\oplus\dots\oplus U_{1,n_k}\oplus U_{f,f}$ for some $k,f\geq 0$.
 \item $M$ is simple and $M=U_{m,n-f}\oplus U_{f,f}$ for some $0\leq f\leq n$.
 \end{enumerate}
\end{thm}
\begin{proof}

First of all we will see that if $M$ satisfies one of the three
conditions then $J$ has a linear projective resolution:

(i) If $M$ has a loop $\{i\}$, then $\partial e_i=1\in J$ so $J$ is
the whole ring $E$ which has a linear resolution.

If $M$ satisfies (ii) then the circuits of $M$ are the circuits of
the $U_{1,n_i}$. Thus all circuits of $M$ have cardinality two which
means that $J$ is generated by linear forms $v_1,\ldots,v_s$ and has
the Cartan complex $C\mlpnt(v_1,\ldots,v_s;E)$ as a linear
resolution.

(iii) Following the remark preceding this theorem  we may
assume that $M$ has no singleton components, so we have $M=U_{m,n}$.
If $m=0$ or $m=n$ then the Orlik-Solomon ideal is $E$ or zero and
has a linear resolution. By Example \ref{ex umn 2} the matroid
$U_{m,n}$ is connected if
$m\not=0,n$. Hence it follows from Theorem \ref{depth} that
$\depth  E/J=1$. By Proposition \ref{ei regular} the variable $e_1$
is $E/J$-regular. The Orlik-Solomon ideal $J=J_{m,n}=(\partial e_A :
|A|=m+1, 1\in A)$ of $U_{m,n}$ was computed in Example \ref{ex umn}.
Then $J+(e_1)=(e_A: |A|=m)+(e_1)$ and thus $J$ reduces modulo $e_1$
to the $m$-th power of the maximal ideal in the exterior algebra
$E/(e_1)$ and hence has a linear projective resolution by Theorem
\ref{linear resolution}.

Now let $J$ have an $m$-linear projective resolution. If $M$ has a
loop, then this is a circuit of cardinality one whence $m=0$. Thus
$M$ satisfies (i).

Now we consider the case that $M$ is simple. As above we assume that
$M$ has no singleton components. So we have to show that
$M=U_{m,n}$. Theorem \ref{linrelations} implies that $M$ is
connected. Then $\depth  E/J=1$ by Theorem \ref{depth} and
$e_1$ is a maximal regular sequence on $E/J$ by Proposition \ref{ei
regular}. Reducing $J$ modulo $(e_1)$ gives the $m$-th power of the
maximal ideal of the exterior algebra $E/(e_1)$ by Theorem
\ref{linear resolution}.

Let $A\subseteq [n]$ with $1\in A$, $|A|=m+1$ and let $A'=A\setminus
\{1\}$. The degree of the residue class of $e_{A'}$ in $E/(e_1)$ is
$m$ and so $\overline{e_{A'}}\in J+(e_1)/(e_1)$. Thus there exists a
representation
$$ e_{A'}=f+ge_1 \qquad f\in J, g\in E.$$
Then
$$
e_A=\pm e_{A'}e_1=\pm fe_1\in J
$$
which is the case if and only if $A$ is dependent. So every subset
of cardinality $m+1$ containing $1$ is dependent.
An analogous argument for $i>1$ shows that every subset of cardinality
$m+1$ is dependent. No subset of cardinality $\leq m$ is dependent
because $J_j=0$ for $j<m$. Thus we conclude $M=U_{m,n}$.

Finally we assume that $M$ has no loops or singleton components, but
non-trivial parallel classes. Then there exists at least one circuit
with two elements. As $J$ is generated in degree $m$ this implies
$m=1$. Let $J_1,\ldots,J_k$ be the Orlik-Solomon ideals of the
components $M_1,\ldots,M_k$ of $M$, i.e. $J=J_1+\ldots+J_k$. Each
$J_j$ is generated by linear forms, because no $\partial e_C$ with
$C$ of one component can be represented by elements $\partial
e_{C_i}$ with $C_i$ of other components. Ideals generated by linear
forms have the Cartan complex with respect to these linear forms as
minimal graded free resolution and this is a linear resolution. Thus
$J_j$ has a linear resolution. It is the Orlik-Solomon ideal of the
connected loopless matroid $M_j$. Following the argumentation in the
preceding paragraph for simple matroids this implies $M_j=U_{1,n_j}$
with $n_j$ the cardinality of the $j$-th component of $M$ and
$M=\Dirsum_{j=1}^kU_{1,n_j}$.
\end{proof}

Since the powers of the maximal ideal of $E$ are strongly stable,
their minimal resolution and especially their Betti numbers are
known from \cite{AHH}. Also
Eisenbud, Fl\o{}ystad and Schreyer give in \cite[Section 5]{EFS} an
explicit description of the minimal graded free resolution of the power of
the maximal ideal using Schur functors. Their result gave the hint
how a ``nicer'' formula of the Betti numbers could look like.

\begin{prop}
The graded Betti numbers of $\mm^t$ are
$$\beta_{i,i+t}(\mm^t)=\binom{n+i}{t+i}\binom{t+i-1}{i}
\text{ and }
\beta_{i,i+j}(\mm^t)=0 \text{ for }j\not=t.
$$

\end{prop}
\begin{proof}
There are $\binom{k-1}{t-1}$ monomials of degree $t$ whose highest
supporting variable is $e_k$, i.e. $m_k(\mm^t)=|\{u\in
G(\mm^t):\max(u)=k\}|=\binom{k-1}{t-1}$. Hence by Lemma
\ref{betti} we obtain
\begin{eqnarray*}
\beta_{i,i+t}(\mm^t)&=& \sum_{k=t}^n m_k(\mm^t)\binom{k+i-1}{k-1}= \sum_{k=t}^n\binom{k-1}{t-1}\binom{k+i-1}{k-1}.
\end{eqnarray*}
That this sums equals $\binom{n+i}{t+i}\binom{t+i-1}{i}$ can be seen
by an induction on $n$, where the induction step from $n$ to $n+1$
is the following:
\begin{eqnarray*}
 &&\sum_{k=t}^{n+1}\binom{k-1}{t-1}\binom{k+i-1}{k-1}\\
&=&\binom{n+i}{t+i}\binom{t+i-1}{i}+\binom{n}{t-1}\binom{n+i}{n}\\
&=&\left(\binom{n+i+1}{t+i}-\binom{n+i}{t+i-1}\right)\binom{t+i-1}{i}+\binom{n}{t-1}\binom{n+i}{n}\\
&=&\binom{n+i+1}{t+i}\binom{t+i-1}{i}.
\end{eqnarray*}
\end{proof}

Now we obtain:

\begin{thm}
\label{Betti_extreme_cases}
Let $M$ be a matroid and $J=J(M)$ be its Orlik-Solomon ideal.
\begin{enumerate}
\item If $M=U_{m,n-f}\oplus U_{f,f}$ for some $f\geq 0$, then $$\beta^E_i(J)=\binom{n-f-1+i}{m+i}\binom{m+i-1}{i}.$$
\item
If $M=U_{1,n_1}\oplus\dots\oplus U_{1,n_k}\oplus U_{f,f}$ for some $k,f\geq 0$, $n=f+\sum_{i=1}^kn_i$, then
$$\beta^E_i(J)=\binom{n-f-k+i}{i+1}.$$
\end{enumerate}
\end{thm}
\begin{proof}
Observe that reducing modulo a regular sequence does not change the Betti numbers
so the Betti numbers of $\mm^m$ give the Betti numbers of $J_{m,n}$.

(i) The Betti numbers of $J$ are the same as the Betti numbers of
the $m$-th power of the maximal ideal in the exterior algebra
$\overline{E}$ on $n-f-1$ variables:
$$\beta^E_i(J)=\beta^{\overline{E}}_i(\overline{\mm}^m)=\binom{n-f-1+i}{m+i}\binom{m+i-1}{i}.$$

(ii) In this case $J$ reduces to the maximal ideal in the exterior
algebra on $n-f-k$ variables because for each component $U_{1,n_i}$
one reduces modulo one variable as in Remark \ref{max reg seq}.
\end{proof}

\section{Examples}
\label{examples}

In this section we study some examples of matroids with small rank
or small number of elements.

Oxley enumerates in \cite[Table 1.1]{Ox} all non-isomorphic matroids
with three or fewer elements. The only loopless matroids among them
are the uniform matroids $U_{1,1}$, $U_{1,2}$, $U_{2,2}$, $U_{1,3}$,
$U_{2,3}$ and $U_{3,3}$. Their depth, complexity and regularity were
already computed in Example \ref{ex umn 2}.

Now we turn to matroids defined by central hyperplane arrangements
in $\C^l$ with $l\leq3$. The arrangement is called central if the
common intersection of all hyperplanes is not empty. A set of $t$
hyperplanes defines an independent set if and only if their
intersection has codimension $t$. Thus every two hyperplanes in a
central arrangement define an independent set and so the matroids
defined by central hyperplane arrangements are simple.

In $\C^1$ the only central hyperplane arrangement consists of a
single point, thus the underlying matroid is $U_{1,1}$.

In $\C^2$ a central hyperplane arrangement consists of $n$ lines
through the origin. The underlying matroid is $U_{2,n}$ if $n\geq 2$
and $U_{1,1}$ if $n=1$.

In $\C^3$ central hyperplane arrangement define various matroids.
One single hyperplane defines a $U_{1,1}$, two hyperplanes a
$U_{2,2}$. Three hyperplanes intersecting in a point give a
$U_{3,3}$, if their intersection is a line then the underlying
matroid is $U_{2,3}$. More generally $n$ hyperplanes through a line
define the matroid $U_{2,n}$. Such an arrangement is called a
\emph{pencil}. For the first time one obtains a matroid that is not
uniform with four hyperplanes taking three hyperplanes intersecting
in a line and a fourth in general position, i.e. the intersection of
the fourth with every two others is a point. The underlying matroid
has two components, one containing the first three hyperplanes and
one singleton component for the fourth hyperplane. It is the matroid
$U_{2,3}\oplus U_{1,1}$. Such an arrangement is an example for a
\emph{near pencil}. For simplicity we define the notions of pencil
and near pencil in terms of their underlying matroid.

\begin{defn}
A central arrangement of $n\geq 3$ hyperplanes is called
\begin{enumerate}
\item a \emph{pencil} if its underlying matroid is $U_{2,n}$.
\item a \emph{near pencil} if its underlying matroid is $U_{2,n-1}\oplus U_{1,1}$.
\end{enumerate}
\end{defn}
In abuse of notation we also call the matroid $U_{2,n}$ a pencil and
$U_{2,n-1}\oplus U_{1,1}$ a near pencil.

A matroid defined by $n$ hyperplanes in $\C^3$ is a simple matroid
of rank 3 unless it is not a pencil which has rank 2. We classify
all simple rank 3 matroids by their connectedness. Then we
determine their homological invariants depth, complexity and
regularity.

It is well-known that a near pencil is the unique
reducible central hyperplane arrangement in $\C^3$; we present a
homological proof for this fact.

\begin{thm}
\label{simple rank 3} Let $M$ be a simple matroid of rank 3. Then
$M$ is connected if and only if it is not a near pencil.
\end{thm}
\begin{proof}
Note that $n\geq 3$ since $M$ has rank 3. If $M=U_{2,n-1}\oplus
U_{1,1}$ is a near pencil, it has two components if $n>3$ and three
components if $n=3$. Thus is it not connected in any case.

Suppose that $M$ has $k$ components with $k>1$ and let $J$ be its
Orlik-Solomon ideal. It is zero if and only if all subsets are
independent. Then $r(M)=3$ implies that $M=U_{3,3}$ is a near
pencil. So from now on we assume $J\not=0$. Since $M$ is simple, $J$
is generated in degree $\geq 2$ and thus $\reg J\geq 2$. Theorem
\ref{depth general} and Corollary \ref{reg=l-k} imply that
$$
\reg J=\reg E/J +1=3-k+1=4-k\leq 2.
$$
Thus the regularity of $J$ is exactly $2$ and $k=2$. Then $J$ has a
2-linear resolution and we may apply Theorem \ref{characterisation}
which says that $M=U_{m,n-i}\oplus U_{i,i}$ for some $0\leq m,i\leq
n$. We may assume $m<n-i$ otherwise $M$ is $U_{3,3}$ and has three
components. Since $M$ is simple, $m$ must be at least 2. Then
$3=r(M)=m+i$ so $i$ can only take the values $0$ or $1$. If $i=0$
then $M=U_{3,n}$ has one or three (if $n=3$) components, so this
case cannot occur. Hence $i=1$ and $M=U_{2,n-1}\oplus U_{1,1}$ is a
near pencil.
\end{proof}

In the following table we have collected the homological invariants
investigated in this paper of all simple ma\-troids of rank 3
which are given by the above
Theorem \ref{simple rank 3}, using \cite[Theorem 3.2]{AAH}, Theorem
\ref{depth general} and Corollary \ref{reg=l-k}. It is a
generalization of Proposition 4.6 of Schenck and Suciu in \cite{SS05},
even including the special case $n=3$.
\newline
\begin{center}
\begin{tabular}{c|ccc}
                    & $\depth E/J$ & $\cx E/J$ & $\reg E/J$\\
\hline
no near pencil      & 1             & $n-1$        & 2\\
\hline
near pencil, $n>3$  & 2             & $n-2$        & 1\\
\hline
near pencil, $n=3$  & 3             & 0          & 0
\end{tabular}
\end{center}
The number of simple rank 3 matroids is e.g.\ determined in \cite{Du}. If
$n=4$ there exist only two simple rank 3 matroids, namely $U_{3,4}$
and $U_{2,3}\oplus U_{1,1}$. If $n=5$ there exist 4 simple rank 3
matroids, $U_{3,5}$, $U_{2,4}\oplus U_{1,1}$ and two further which
cannot be expressed as sum of uniform matroids since they must be
connected by Theorem \ref{simple rank 3}. One is the underlying
matroid of an arrangement of five hyperplanes, three intersecting in
a line and two in general position to each other and to the first
three hyperplanes. The matroid has only one circuit with three
elements corresponding to the first three hyperplanes and three
circuits with four elements. The arrangement of five hyperplanes
defining the second matroid has twice three hyperplanes intersecting
in a line. The matroid has two circuits with three elements
corresponding to these triples, and one circuit with four elements,
not containing the element in the intersection of the other
circuits.

\end{document}